\documentclass
{article}
\usepackage{amsmath, amssymb,amsfonts, amsxtra, latexsym, amscd,amsthm, 
pb-diagram,  graphics}
\usepackage [all]{xy}

\theoremstyle{definition}
\theoremstyle{plain}

\newcommand{\A}{\mathcal{A}}
\newcommand{\tx}{\otimes }
\newcommand{\ts}{\oplus}
\newcommand{\ri}{\rightarrow }

\newcommand{\Lh}{\frak{L}}
\newcommand{\Rh}{\frak{R}}
\newcommand{\Fm}{\widetilde{F}}

\newcommand{\Lo}{\hat{L}}
\newcommand{\Ro}{\hat{R}}
\newcommand{\Gm}{\widetilde{G}}
\newcommand{\Hm}{\widetilde{H}}
\newcommand{\Km}{\widetilde{K}}

\newcommand{\Fa}{\Breve{F}}
\newcommand{\Ka}{\Breve{K}}
\newcommand{\La}{\Breve{L}}
\newcommand{\Ra}{\Breve{R}}

\newcommand{\Ga}{\Breve{G}}
\newcommand{\Ha}{\Breve{H}}

\newtheorem{thm}{\bf Theorem}[section]
\newtheorem{lem}[thm]{\bf Lemma} 
\newtheorem{pro}[thm]{\bf Proposition}
\newtheorem{hq}[thm]{\bf Corollay}
\newtheorem{Note}[thm]{\bf Remark} 
\newtheorem{dn}[thm]{Definition}

\begin{document}

\title{ \Large{\bf Cohomological classification of braided $Ann$-categories}}

\pagestyle{myheadings} 

\markboth{Nguyen Tien Quang-Dang Dinh Hanh}{Cohomological classification of braided $Ann$-categories}

\maketitle

\centerline{\MakeUppercase{Nguyen Tien Quang} and \MakeUppercase{Dang Dinh Hanh}}

\vspace{0.3cm}

\centerline{\it Dept. of Mathematics, Hanoi National University of Education,}

\vspace{0.1cm}

\centerline{\it 136 Xuanthuy, Caugiay, Hanoi, Vietnam }

\vspace{0.3cm}

{\centerline{\it Email: nguyenquang272002@gmail.com,\ \  \ \ ddhanhdhsphn@gmail.com }}

\vspace{0.3cm}

\setcounter{tocdepth}{1}

\begin{abstract}  
A braided $Ann$-category $\mathcal A$ is an $Ann$-category $\mathcal A$ together with a braiding $c$ such that $(\mathcal A, \otimes, a, c, (1,l,r))$ is a braided tensor category, moreover $c$ is compatible with the distributivity constraints. According to the structure transport theorem, the paper shows that each braided $Ann$-category is equivalent to a braided $Ann$-category of the type $(R,M)$, hence the proof of the classification theorem for braided $Ann$-categories by the cohomology of commutative rings is presented.
\end{abstract}

\ \ \ \parbox[t]{11cm}{\small{\bf Mathematics Subject Classifications (2000):} 18D10, 16E40}

\vspace{0.3cm}

\ \ \ \parbox[t]{11cm}{{\small{\bf Keywords:} Braided $Ann$-category, structure transport, classification theorem,  cohomology of rings.}}

\section{Introduction}
The notion of monoidal categories or tensor categories was presented by S. Mac Lane \cite{Mac3}, J. B\'{e}nabou \cite{Be} in 1963. Then, these categories were ``refined" to become  {\it categories with group structure}, when the notion of  invertible objects  was added (see M. L. Laplaza \cite{Lap2}, S. Rivano \cite{Ri}). Now, if the underlying category is a {\it groupoid} (i.e., all morphisms are isomorphisms), we obtain the concept of a {\it monoidal category group-like} (see A. Fr\"ohlich and C. T. C. Wall \cite{FroWa}), or a Gr-category (see H. X. Sinh \cite{Si}), or recently called {\it categorical group}. The variety of categorical groups has been classified by the cohomology group   $H^{3}(G, A)$ (see \cite{Si}).

The concept of a braided tensor category was introduced by A. Joyal and R. Street \cite{JS} which is a necessary extension of a symmetric  tensor category, since the center of a tensor category is  a braided tensor category but unsymmetric. In \cite{JS}, these authors have classified the variety of braided categorical groups by the category of quadratic functions (thanks to a result of S. Eilenberg and S. Mac Lane on representations of quadratic functions by the abelian cohomology group $H^{3}_{ab}(G, A)$ \cite{EM}, \cite{Mac1}). Before that, the case of symmetric categorical groups (or Picard categories) was solved by H. X. Sinh \cite{Si}.

A more general situation for  Picard categories was given by A. Fr\"ohlich and C. T. C. Wall
with the name {\it graded categorical groups} \cite{FroWa} (which is later called graded Pic-categories by A. Cegarra and E. Khmaladze \cite{CK1}). Homotopical classification  theorems for the variety of graded categorical groups, the variety of braided graded categorical groups, and its particular case, the variety of graded Picard categories have been presented, respectively, in \cite{Gar-Del}, \cite{CK1}, \cite{CK2}. Each category raises to a 3-cocycle in some sense that each congruence class of the same categories is corresponding to a third-dimensional cohomology class.

The categories with two monoidal structures interest many authors. In 1972, M. L. Laplaza \cite{Lap} studied   distributivity categories. The main result of \cite{Lap} is a proof of coherence theorem for these categories. Later, in \cite{FroWa}, A. Fr\"{o}hlich and C. T. C Wall presented  the concept of {ring-like categories} with intention is to offer a new axiomatics which is shorter than the one of M. L. Laplaza  \cite{Lap}.  These two concepts  are formalizations of the category of the modules over a commutative ring. 

In 1994, M. Kapranov and V. Voevodsky \cite{KapVo} omitted requirements of the axiomatics of Laplaza which are related to the commutativity constraints of the operation $\otimes$ and presented the name {\it ring categories} to indicate these categories.

In order to obtain descriptions on structures as well as cohomological classification, N. T. Quang introduced  the concept of $Ann$-categories, 
as a categorification of the concept of rings,  with requirements of invertibility of objects and morphisms of the underlying category, similar to the case of categorical groups (see  \cite{Gar-Del}). These additional requirements are not too special, because if $\mathcal P$ is a Picard category, then the category End($\mathcal P$) of Pic-functors over $\mathcal P$ is an $Ann$-category (see \cite{Q11}), this was repeated in \cite{HCCZ1}. Besides, each $Ann$-category is a ring category \cite{QTP}. Moreover, we can prove that each congruence class of $Ann$-categories is completely defined by three invariants: the ring $R$, the $R$-bimodule $M$ and an element in the  Mac Lane cohomology group  $H^{3}_{MacL}(R, M)$ (see \cite{Q3}). Regular $Ann$-categories (whose commutativity constraint satisfies $c_{X,X}= id$ for any object $X$) were classified by the Shukla cohomology group $H^{3}_{Sh}(R, M)$ (see \cite{Q12}). $Ann$-functors were also classified by the low-dimensional Mac Lane cohomology groups \cite{Q10}.

In 2006, M. Jibladze and T. Pirashvili \cite{JiPi} presented the concept of  {\it categorical ring}  as a slightly modified version of the concept of $Ann$-categories. Hence, in \cite{JiPi}, they classified categorical rings by the third cohomology group of ring due to Mac Lane. In \cite{QHT}, we proved that the variety of categorical rings contained the variety of $Ann$-categories. Conversely, from the axiomatics of the categorical rings may not deduce the structure of $\pi_0\A$-bimodule of $\pi_1\A$. Each $Ann$-category can be regarded as an one-object of Gpd-categories in the thesis of M. Dupont \cite{Dup}, or an one-point enrichment of SPC of V. Schmitt \cite{Sch}.

The study of the braiding in an $Ann$-category is a natural process. In \cite{Q2}, it is proved that the two concepts of ``distributivity category'' of M. L. Laplaza  and ``ring-like category'' of A. Fr\"ohlich and C. T. C Wall coincide and this concept is broader than the one of symmetric $Ann$-categories.

The purpose of this paper is to classify the braided $Ann$-categories and to expand the well-known results of $Ann$-categories. First, we showed the equivalence of each braided $Ann$-category with a braided $Ann$-category of the type $(R, M)$. Moreover, 
this equivalence induces  a functor
$${\bf BrAnn}\rightarrow {\bf H^{3}_{BrAnn}} $$
between the category of braided $Ann$-categories and the  category of  3-cohomology classes of commutative rings. 

 For any commutative ring $R$ and $R$-module $M,$ we define the cohomology groups $H^{n}_{ab}(R,M),$ where $n=1,2,3$ and prove that there exists a bijection 
$$ {\bf BrAnn}[R,M]\leftrightarrow  H^{3}_{ab}(R,M),$$
 where $ {\bf BrAnn}[R,M]$ is the set of classes of $Ann$-categories of the type $(R,M).$


In particular, this cohomology class is just an element of the Harrison cohomology group $H^{3}_{Har}(R, M).$

In this paper, we sometimes denote  $XY$ instead of  $X\tx Y$ for two objects  $X, Y$.


\section{Cohomology for rings}
\subsection{Mac Lane cohomology groups for rings}

Let $R$ be a ring and $M$ be an $R$-bimodule. From the definition of ring cohomology of S. Mac Lane \cite{Mac2}, we may obtain the description of elements of the cohomology group $H^3_{MacL}(R,M)$ (see Proposition 7.2, Proposition 7.3 \cite{Q3}).

The group $Z^3_{MacL}(R,M)$ of 3-cocycles of the ring $R$ with coefficients in the $R$-bimodule $M$ consists of 5-tuple $(\xi,\eta,\alpha,\lambda,\rho)$ of the maps
$$\xi,\alpha,\lambda,\rho:R^3\rightarrow M,\ \ \eta:R^2\rightarrow M$$ 
which satisfy the following conditions for all $x,y,z,t\in R$:
\begin{eqnarray*}
S1.&
\xi(y,z,t) - \xi(x+y,z,t) +\xi(x,y+z,t) - \xi(x,y,z+t) + \xi(x,y,z) = 0,\\
S2.&\xi(x,y,z) - \xi(x,z,y) +\xi(z,x,y) + \eta(x+y,z) - \eta(x,z) - \eta(y,z) = 0,\\
S3.&\eta(x,y) + \eta(y,x) = 0,\\
S4.&x\eta(y,z) - \eta(xy,xz) = \lambda(x,y,z) - \lambda(x,z,y),\\
S5.&\eta(x,y)z - \eta(xz,yz) =  \rho(x,y,z) - \rho(y,x,z),\\
S6.&x\xi(y,z,t) - \xi(xy,xz,xt) =\lambda(x,z,t) - \lambda(x,y+z,t) + \lambda(x,y,z+t)\\
& - \lambda(x,y,z),
\end{eqnarray*}
\begin{eqnarray*}
S7.& \xi(x,y,z)t - \xi(xt,yt,zt) = \rho(y,z,t) - \rho(x+y,z,t)+\rho(x,y+z,t)\\
&-\rho(x,y,z),\\
S8.&\rho(x,y,z+t) - \rho(x,y,z) -\rho(x,y,t) +\lambda(x,z,t)+\lambda(y,z,t)\\
& - \lambda(x+y,z,t) =
\xi(xz+xt,yz,yt)+\xi(xz,xt,yz)-\eta(xt,yz)\\
& +\xi(xz+yz,xt,yt)-\xi(xz,yz,xt),\\
S9.&\alpha(x,y,z+t)-\alpha(x,y,z)-\alpha(x,y,t) = x\lambda(y,z,t) + \lambda(x,yz,yt)\\
&-\lambda(xy,z,t),\\
S10.&\alpha(x,y+z,t)-\alpha(x,y,t)-\alpha(x,z,t) = x\rho(y,z,t) - \rho(xy,xz,t)+\\
&\lambda(x,yt,zt) -\lambda(x,y,z)t,\\
S11.&\alpha(x+y,z,t) - \alpha(x,y,t) -\alpha(y,z,t) = -\rho(x,y,z)t - \rho(xz,yz,t)\\
& + \rho(x,y,zt),\\
S12.&x\alpha(y,z,t) - \alpha(xy,z,t) + \alpha(x,yz,t) - \alpha(x,y,zt) + \alpha(x,y,z)t = 0,
\end{eqnarray*}
\noindent and satisfy the normalized conditions:
\begin{eqnarray*}
\xi(0,y,z)&=&\xi(x,0,z) = \xi(x,y,0) = 0,\\
\alpha(1,y,z)&=&\alpha(x,1,z) = \alpha(x,y,1) = 0,\\
\alpha(0,y,z)&=&\alpha(x,0,z) = \alpha(x,y,0) = 0,\\
\lambda(1,y,z)&=&\lambda(0,y,z) = \lambda(x,0,z) = \lambda(x,y,0) = 0,\\
\rho(x,y,1)&=&\rho(0,y,z) = \rho(x,0,z) = \rho(x,y,0) = 0.
\end{eqnarray*} 

\noindent The subgroup  $B^3_{MacL}(R,M)\subset Z^3_{MacL}(R,M)$ of 3-coboundaries consists of the 5-tuple of maps $h=(\xi,\eta,\alpha,\lambda,\rho)$ such that there exist the maps $g=(\mu,\nu):R^2\rightarrow M$ satisfying \ $h=\partial_{MacL} g$, i.e.,
\[\begin{array}{llllllllllllll}
S13.& \xi(x,y,z) & = & \mu(y,z)- \mu(x+y,z) + \mu(x,y+z) - \mu(x,y),\\
S14.& \eta(y,x) & = & \mu(x,y) - \mu(y,x),\\
S15.& \alpha(x,y,z) & = & x\nu(y,z) - \nu(xy,z) + \nu(x,yz) - \nu(x,y)z,\\
S16.& \lambda(x,y,z) & = &\nu(x,y+z) - \nu(x,y) - \nu(x,z) + x\mu(y,z) -\mu(xy,xz),\\
S17.& \rho(x,y,z) & = &\nu(x+y,z) - \nu(x,z) - \nu(y,z) + \mu(x,y)z - \mu(xz,yz),
\end{array}\]

\noindent where $\mu,\nu$ satisfy the ``normalized'' condition $\mu(0,y) = \mu(x,0) = 0$ and\linebreak $\nu(0,y) = \nu(x,0) = \nu(1,y) = \nu(x,1) = 0.$

The group $Z^2_{MacL}(R,M)$ of 2-cocycles $g=(\mu,\nu),$ of the ring $R$ with coefficients in the $R$-bimodule $M,$ satisfying:

\[\partial_{MacL} g=0.\]

\noindent The subgroup  $B^2_{MacL}(R,M)\subset Z^2_{MacL}(R,M)$ of 2-coboundaries are the pairs $(\mu,\nu)$  such that there exists a map $t:R\rightarrow M$ satisfying $(\mu,\nu)=\partial_{MacL}t$, i.e.,
\vspace{0.2cm}

\noindent $S18.\ \ \mu(x,y) =  t(y)-t(x+y)+t(x),$

\vspace{0.15cm}

\noindent $S19.\  \ \nu(x,y)  =  xt(y) - t(xy)+t(x)y,$

\vspace{0.2cm}
\noindent where $t$ satisfies the normalized condition $t(0)=t(1)=0$. 

The group $Z^1_{MacL}(R,M)$   consists of 1-cochains $t,$ of the ring $R$ with coefficients in the $R$-bimodule $M,$ satisfying:
\[\partial_{MacL} t=0.\]

\noindent The subgroup $B^1_{MacL}(R,M)\subset Z^1_{MacL}(R,M)$ consists of 1-coboundaries which are maps $t$ such that there exists $a\in R$ satisfying $t(x)=ax-xa$.

The quotient groups
\[H^i_{MacL}(R,M)=Z^i_{MacL}(R,M)/ B^i_{MacL}(R,M),\ i=1,2,3 \]
are called the {\it $i$-th Mac Lane cohomology group}  of the ring $R$ with coefficients in the $R$-bimodule $M$.


\subsection{The low-dimensional cohomology groups of a comutative ring}

When the ring $R$ is commutative and $M$ is an $R$-module, we define the cohomology groups $H^n_{ab}(R,M)\ (n=1, 2, 3)$  as follows, note that $M$ is regarded as an $R$-bimodule with the two-sided actions  of the ring $R$ on $M$ coincide.

The group of abelian 3-cocycles $Z^3_{ab}(R,M)$ consists of pairs  $(h, \beta)$, where $h\in Z^3_{MacL}(R,M)$ and the map $\beta: R^2\rightarrow M$ satisfies the following conditions:
\begin{eqnarray*}
S20. &\alpha(x,y,z)-\alpha(x,z,y)+\alpha(z,x,y)+x\beta(y,z)-\beta(xy,z)+y\beta(x,z)=0,\\
S21. &\alpha(x,y,z)-\alpha(y,x,z)+\alpha(y,z,x)-y\beta(x,z)+\beta(x,yz)-z\beta(x,y)=0,\\
S22. &\beta(x,y)-\beta(x,y+z)+\beta(x,z)=\rho(y,z,x)-\lambda(x,y,z).
\end{eqnarray*}

The group of abelian 3-coboundaries $B^3_{ab}(R,M),$ which is a subgroup of $Z^3_{ab}(R,M),$ consists of the pairs  $(h, \beta)$, such that there exists a pair $(\mu, \nu):R^2\rightarrow M$ satisfying the normalized conditions and the relations 
$h=\partial_{MacL}(\mu,\nu)\in B^3_{MacL}(R,M)$, and $\beta(x,y)=\nu(x,y)-\nu(y,x)$. Then, we write  $(h, \beta)=\partial(\mu,\nu).$

The group $Z^2_{ab}(R,M)$ consists of the pairs  $(\mu, \nu):R^2\rightarrow M$, such that\linebreak $(\mu, \nu)\in Z^2_{MacL}(R,M)$, moreover $\nu$ satisfies the relation
\[\nu(x,y)-\nu(y,x)=0.\]

The subgroup 
$B^2_{ab}(R,M)$ of $Z^2_{ab}(R,M)$ consists of the pairs $(\mu, \nu)$ such that $(\mu, \nu)\in B^2_{MacL}(R,M)$.

The quotient groups 
\[H^i_{ab}(R,M)=Z^i_{ab}(R,M)/B^i_{ab}(R,M),\ i=2,3 \]
are called the {\it $i$-th commutative cohomology groups}  of the commutative ring $R$ with coefficients in the $R$-module $M$.
 
Finally, we define
\[H^1_{ab}(R,M)=H^1_{MacL}(R,M).\]


\section { Braided $Ann$-categories}

\subsection{Definitions and Examples}

First, let us recall the notion of a {\it braided monoidal category} according to \cite{JS}.

{\it A \emph{braiding} for a monoidal category $\mathcal V$ consists of a natural family of isomorphisms \[c=c_{A,B}: A\otimes B \stackrel{\sim}{\longrightarrow} B\otimes A\]
in $\mathcal V$ such that the two diagrams} (B1) {\it and} (B2) {\it commute:}

\[\scriptsize\begin{diagram}
\node{(A\otimes B)\otimes C}\arrow{e,t}{c\tx id}\arrow{s,l}{a^{-1}}
\node{(B\otimes A)\otimes C}\arrow{e,t}{a^{-1}}
\node{B\otimes (A\otimes C)}\arrow{s,r}{id\tx c}\\
\node{A\otimes (B\otimes C)}\arrow{e,t}{c}
\node{(B\otimes C)\otimes A}\arrow{e,t}{a^{-1}}
\node{B\otimes (C\otimes A)}\tag{B1}
\end{diagram}\]
\[\scriptsize\begin{diagram}
\node{A\otimes (B\otimes C)}\arrow{e,t}{id\tx c}\arrow{s,l}{a}
\node{A\otimes (C\otimes B)}\arrow{e,t}{a}
\node{(A\otimes C)\otimes B}\arrow{s,r}{c\tx id}\\
\node{(A\otimes B)\otimes C}\arrow{e,t}{c}
\node{C\otimes (A\otimes B)}\arrow{e,t}{a}
\node{(C\otimes A)\otimes B}\tag{B2}
\end{diagram}\]

If $c$ is a braiding,  so is $c'$ given by  $c'_{A,B}=(c_{B,A})^{-1}$,  since (B2) is just obtained from (B1) by replacing $c$ with $c'$. A {\it symmetry} is a braiding $c$ which satisfies  $c'=c$. 

{\it A \emph{braided monoidal category} is a pair $(\mathcal V, c)$ consisting of a monoidal category $\mathcal V$ and a braiding $c$.}

\begin{dn}[{\cite{Q11}}]
An \emph{$Ann$-category} consists of:
\begin{enumerate}
\item[(i)] A category $\mathcal A$ together with two bifunctors $\oplus,\otimes:\mathcal A \times \mathcal A\longrightarrow \mathcal A.$
\item[(ii)] A fixed object $0\in \mathcal A$ together with natural constraints $a^+,c^+,g,d$ such that $(\mathcal A,\oplus,a^+,c^+,(0,g,d))$ is a symmetric categorical group.
\item[(iii)] A fixed object $1\in\mathcal A$ together with natural constraints $a,l,r$ such that $(\mathcal A,\otimes,a,(1,l,r))$ is a monoidal category.
\item[(iv)] Natural isomorphisms $\frak{L},\frak{R}$
\begin{equation*}
\begin{array}{cccc}
\frak{L}_{A, X, Y} : &A\otimes (X\oplus Y) &\rightarrow& (A\otimes X)\oplus (A\otimes Y),\\
\frak{R}_{X, Y, A}:&(X\oplus Y)\otimes A &\rightarrow &(X\otimes A)\oplus (Y\otimes A),
\end{array}
\end{equation*}
such that the following conditions are satisfied:
\end{enumerate}

(Ann-1) For each object $A\in \mathcal A,$ the pairs $(L^A,\breve{L^A}),(R^A,\breve{R^A})$ defined by the relations:
\begin{equation*}
\begin{array}{cccc}
L^A = A \otimes -,&\qquad\qquad&R^A = - \otimes A,\\
\Breve{L}_{X, Y}^A = \frak{L}_{A, X, Y},&\qquad\qquad&\Breve{R}_{X, Y}^A = \frak{R}_{X, Y, A}
\end{array}
\end{equation*}
 are $\oplus$-functors which are compatible with $a^+$ and $c^+$.

(Ann-2) For all objects $ A,B,X,Y\in \mathcal A,$ the following diagrams:
\[\scriptsize\begin{diagram}
\node{(AB)(X\oplus Y)}\arrow{s,l}{\Breve{L}^{AB}}
\node{A(B(X\oplus Y))}\arrow{e,t}{id_A \otimes \Breve{L}^B}\arrow{w,t}{\ a_{A, B, X\oplus Y}}
\node{A(BX\oplus BY)}\arrow{s,r}{\Breve{L}^A}\\
\node{(AB)X\oplus (AB)Y}
\node[2]{A(BX)\oplus A(BY)}\arrow[2]{w,t}{a_{A, B, X}\oplus a_{A, B, Y}}
\tag{1}\label{bd1}
\end{diagram}\]
\[\scriptsize\begin{diagram}
\node{(X\oplus Y)(BA)}\arrow{s,l}{\Breve{R}^{BA}}\arrow{e,t}{a_{X\oplus Y, B, A}}
\node{((X\oplus Y)B)A}\arrow{e,t}{\Breve{R}^B \otimes id_A}
\node{(XB\oplus YB)A}\arrow{s,r}{\Breve{R}^A}\\
\node{X(BA)\oplus Y(BA)}\arrow[2]{e,t}{a_{X, B, A}\oplus a_{Y, B, A}}
\node[2]{(XB)A\oplus (YB)A}
\tag{2}\label{bd2}
\end{diagram}\]
\[\scriptsize\begin{diagram}
\node{(A(X\oplus Y))B}\arrow{s,l}{\Breve{L}^A \otimes id_B}
\node{A((X\oplus Y)B)}\arrow{w,t}{a_{A, X\oplus Y, B}}\arrow{e,t}{id_A \otimes \Breve{R}^B}
\node{A(XB\oplus YB)}\arrow{s,r}{\Breve{L}^A}\\
\node{(AX\oplus AY)B}\arrow{e,t}{\Breve{R}^B}
\node{(AX)B\oplus (AY)B}
\node{A(XB)\oplus A(YB)}\arrow{w,t}{a\ \oplus\ a}
\tag{3}\label{bd3}
\end{diagram}\]

\[\scriptsize
\begin{diagram}
\node{(A\oplus B)X\oplus (A\oplus B)Y}\arrow{s,l}{\Breve{R}^X \oplus \Breve{R}^Y}
\node{(A\oplus B)(X\oplus Y)}\arrow{w,t}{\Breve{L}^{A\oplus B}}\arrow{e,t}{\Breve{R}^{X\oplus Y}}
\node{A(X\oplus Y)\oplus B(X\oplus Y)}\arrow{s,r}{\Breve{L}^A \oplus \Breve{L}^B}\\
\node{(AX\oplus BX)\oplus (AY\oplus BY)}\arrow[2]{e,t}{v}
\node[2]{(AX\oplus AY)\oplus (BX\oplus BY)}
\tag{4}\label{bd4}
\end{diagram}\]
commute, where $v = v_{_{U, V, Z, T}}: (U\oplus V)\oplus (Z\oplus T) \rightarrow (U\oplus Z)\oplus (V\oplus T)$ \ 
is the unique morphism built from $a^+,c^+,id$ in the symmetric categorical group $(\mathcal A,\oplus).$

(Ann-3) For the unit object $1\in \mathcal A$ of the operation $\otimes,$ the following diagrams:
\[\scriptsize\begin{diagram}
\node{1(X\oplus Y)}\arrow[2]{e,t}{\Breve{L}^1}\arrow{se,b}{l_{X\oplus Y}}
\node[2]{1X \oplus 1Y}\arrow{sw,b}{l_X\oplus l_Y}
\node{(X\oplus Y)1}\arrow[2]{e,t}{\Breve{R}^1}\arrow{se,b}{r_{_{X\oplus Y}}}
\node[2]{X1\oplus Y1}\arrow{sw,b}{r_{_X}\oplus r_{_Y}}\\
\node[2]{X\oplus Y}
\node{(5.1)}
\node[2]{X\oplus Y}
\node{(5.2)}
\end{diagram}
\]
commute.
\end{dn}

 \begin{dn} 
 A \emph{braided $Ann$-category}  $\mathcal A$ is an $Ann$-category $\mathcal A$ together with a  braiding $c$ such that $(\mathcal A, \otimes, a, c, (1,l,r))$ is a braided monoidal category, and $c$ makes the following diagram 
 \[\scriptsize\begin{diagram}\label{bd5}
\node{A(X\ts Y)}\arrow{e,t}{\La^A_{X,Y} }\arrow{s,l}{c}
\node{AX\ts AY}\arrow{s,r}{c\ts c}\\
\node{(X\oplus Y)A}\arrow{e,t}{\Ra^A_{X,Y}}
\node{XA\ts YA}
\tag{6}
\end{diagram}\]
commutes and satisfies $c_{_{0, 0}}=id$.

A braided $Ann$-category  is called a \emph{symmetric $Ann$-category} if the braiding $c$ is  symmetric. 
\end{dn}

\noindent {\bf Example 1.} {\bf A braided $Ann$-category of the type $(R,M)$}

Let $R$ be a commutative ring with the unit and $M$ be an $R$-module. For any  normalized 3-cocycle
$h\in Z^{3}_{MacL}(R,M),$ there is an $Ann$-category ${\mathcal S}=(R,M)$ defined as follows: objects of $\mathcal S$ are the elements of $R$, morphisms are automorphisms  $(x,a):x\rightarrow x$.
The composition law of morphisms in ${\mathcal S}$ is the addition operation in  $M$. The two operations $\oplus$ and $\otimes$ of ${\mathcal S}$ are defined by:
\begin{eqnarray*}
x\oplus y=x+y, & (x,a)\oplus (y,b)=(x+y,a+b);\\
x\otimes y=xy,& (x,a)\otimes (y,b)=(xy, xb+ya).
\end{eqnarray*}

The unit constraints for the two operations are strict, the associativity, commutativity constraints of $\oplus$, the associativity constraint of $\otimes$ and distributivity constraints are the maps, respectively, of 3-cocycle $h=(\xi,\eta,\alpha,\lambda,\rho)$. Then $\mathcal S$ is a braided $Ann$-category with the braiding is the map $\beta:R^2\rightarrow M$ satisfies the equations  (S20), (S21) and (S22).

If $\mathcal S$ is a symmetric $Ann$-category, then $\beta$ satisfies the equations (S20), (S22) and the following equation (S21b):\\
$S21b. \ \beta(x,y)+\beta(y,x)=0.$
\vspace{0.3cm}

\noindent {\bf Example 2.} {\bf The center of an almost strict $Ann$-category}

Let $\mathcal A$ be an {\it almost strict $Ann$-category}.
{\it The center of an $Ann$-category} $\mathcal A$, denoted by  $C_{\mathcal A}$, is a category in which objects are pairs $(A,u)$, with $A\in Ob(\mathcal A)$ and
\begin{equation*}
u_{A,X}: A\otimes X \longrightarrow X\otimes A
\end{equation*} 
is a natural transformation satisfying the following three conditions:
\begin{eqnarray*}
\qquad \ \quad\ \  u_{1,X}&=&id_X,\\
\qquad  \ u_{A,X\tx Y}&=&(id_X\tx u_{A,Y})\circ (u_{A,X}\tx id_Y),\\
\qquad  \ u_{A,X\ts Y}&=&(u_{A,X}\ts u_{A,Y})\circ \La_{A,X,Y}.
\end{eqnarray*}

\noindent The morphism $f: (A,u_A)\rightarrow (B,u_B)$ of $C_{\mathcal A}$ is a morphism $f: A\rightarrow B$ of $\mathcal A$ satisfying the condition
 $$\ \  u_{B,X}\circ(f\otimes id)=(id\otimes f)\circ u_{A,Y}.$$ 

Then $C_{\A}$ becomes a braided $Ann$-category with the sum of two objects given by \begin{equation*}
 (A,u_A)+(B,u_B)=(A\ts B, \La_{-,A,B}^{-1}\circ u_{A\ts B});
\end{equation*}
and the sum of two morphisms of $C_{\A}$ is indeed the sum of two morphisms of $\A$; the product of two objects of  $C_{\A}$ given by:
\begin{equation*}
(A,u_A)\times (B,u_B)=(A\otimes B, (u_A\tx id)\tx (id\tx  u_B));
\end{equation*}
the tensor product of two morphisms of  $C_{\A}$ is indeed the tensor product of two morphisms of $\A$, and the braiding $c$ given by:
\[c_{(A,u_A), (B,u_B)}=u_{A,B}: (A,u_A)\tx (B,u_B)\ri (B,u_B)\tx (A,u_A).\]

The detail proof for this example was presented in \cite{Q2}.

\vspace{0.3cm}

\noindent {\bf Example 3.} {\bf A symmetric $Ann$-category of a regular homomorphism}

First, let us recall some events of the ring  extension problem which was presented by 
S. Mac Lane \cite{Mac2} and the $Ann$-category of a regular homomorphism was presented in \cite{Q12}. 

For a given ring $A$, we call $M_A, P_A$, respectively, the ring of bimultiplications, the ring of outer-bimultiplications of the ring $A$, where $P_A=M_A/\tau(A)$, with $\tau:A\rightarrow M_A$ is a ring homomorphism defined by:
\[\tau_ca=ca,\ \ a\tau_c=ac,\ \ \ a, c\in A.\]  
For any regular homomorphism $\theta: R\rightarrow P_A$, {\it the bicenter} of the ring $A$,
$$C_A=\{c\in A\ |\ ac=ca=0,\ \forall a\in A\},$$
is an $R$-bimodule with the actions
\[xc=(\theta x)c,\ \ cx=c(\theta x),\ \ c\in C_A, x\in R.\]

The ring extension problem requires one to find a short exact sequence of rings 
\[0\longrightarrow A\longrightarrow S\longrightarrow R\longrightarrow 0\]
which induces the given homomorphism $\theta: R\rightarrow P_A$. In \cite{Q12}, it is showed that  $\theta$ induces an $Ann$-category $(R, C_A, h),$ it can be shortly described as follows.

Let $\sigma: R\rightarrow M_A$ be a map such that $\sigma(x)\in \theta x,\ x\in R$ and $\sigma(0)=0,$ $\sigma(1)=1$. Then $\sigma$ induces two maps 
$$f, g: R\times R\rightarrow A$$
satisfying
\begin{eqnarray*}
\tau_{ f(x,y)}&=& \sigma(x+y)-\sigma(x)-\sigma(y),\\
\tau_{ g(x,y)}&=& \sigma(xy)-\sigma(x)\sigma(y).
\end{eqnarray*}

The ring structure of $M_A$ induces a 3-cocycle $h\in Z^{3}_{MacL}(R, C_A)$, where $h$ is a family of maps $\xi, \alpha, \lambda, \rho: R^3\rightarrow C_A$ and $\eta: R^2\rightarrow C_A.$  And so, $(R, C_A, h)$ is an $Ann$-category (Proposition 6.1 \cite{Q12}).

Now, if the rings  $A, R$ are commutative, then there exists a map\linebreak
 $\beta:R^2\rightarrow C_A$ 
such that
\begin{equation*}
\tau_{\beta(x,y)}=\sigma(x)\sigma(y)-\sigma(y)\sigma(x),\ \ x,y\in R.
\end{equation*} 
It is easy to see that 
\[\beta(x,y)=g(x,y)-g(y,x).\]
Moreover, one can directly verify that  $\beta$ is a symmetric constraint of the $Ann$-category $(R, C_A, h)$, it means that $(R, C_A, h, \beta)$ is a symmetric $Ann$-category.

We also call the family $(\xi,\eta, \alpha, \beta, \lambda, \rho)$ {\it an obstruction} of the homomorphism $\theta$. When all these maps vanish, then there exists the commutative ring extension $S$ of $R$ by $A$ defined as follows:
\[S=\{(a,r)\ |\ a\in A, \ r\in R\}\] 
together with the operations:
\begin{eqnarray*}
(a_1,r_1)+(a_2,r_2)&=&(a_1+a_2+f(r_1,r_2), r_1+r_2),\\
(a_1,r_1).(a_2,r_2)&=&(r_1a_2+r_2a_1+r_1r_2+g(r_1,r_2), r_1r_2).
\end{eqnarray*}

In the next sections, we shall transport the results on $Ann$-categories to braided $Ann$-categories. First, we extend the technique of structure transport to construct the reduced braided $Ann$-category $s\mathcal A$ of the braided $Ann$-category $\mathcal A$. This category is equivalent to the original braided $Ann$-category, so we can carry out the classification of braided $Ann$-categories on their reductions. Diagrams will be used instead of transformations of equations to make it convenient for readers to follow detail proof.


\subsection{Braided $Ann$-functors and the structure transport}\label{sec3}

It is known that if $(F, \Fa):\A\rightarrow \A'$ is an $\oplus$-functor which is compatible with 
the associativity constraints of two categorical groups, then it is also compatible with the unit constraints, i.e., one can deduce the isomophism   $F_{+}:F(0)\rightarrow 0$ such that  $(F, \Fa, F_{+})$ is a monoidal functor. We also recall that,  the concept of symmetric monoidal functors is not different from the one of braided  monoidal functors. With these notes, we have the following definition.

\begin{dn} Let  $\A$, $\A'$ be braided $Ann$-categories. A \emph{braided $Ann$-functor} is an $Ann$-functor which is compatible with the braidings. In other words, a braided $Ann$-functor from 
$\A$ to $\A'$ is a quadruple $(F, \Fa, \Fm, F_{\ast})$, where $(F, \Fa),$ $(F,\Fm, F_{\ast})$ are symmetric monoidal functors, respectively, for $\oplus,$ $\otimes,$ and satisfies the two following commutative diagrams:
\[\scriptsize
\begin{diagram}\label{bd8.1}
\node{F(X(Y\oplus Z))}\arrow{s,l}{F(\Lh)}
\node{F(X)F(Y\oplus Z)}\arrow{w,t}{\Fm}
\node{F(X)(FY\oplus FZ)}\arrow{s,r}{\Lh'}\arrow{w,t}{id\otimes \Fa}\\
\node{F(XY\oplus XZ)}
\node{F(XY)\oplus F(XZ)}\arrow{w,t}{\Fa}
\node{FXFY\oplus FXFZ}\arrow{w,t}{\Fm \oplus \Fm}
\tag{7.1}
\end{diagram}\]
\[\scriptsize\begin{diagram}\label{bd8.2}
\node{F((X\ts Y) Z)}\arrow{s,l}{F(\Rh)}
\node{F(X\ts Y)FZ}\arrow{w,t}{\Fm}
\node{(FX\oplus FY)FZ}\arrow{s,r}{\Rh'}\arrow{w,t}{ \Fa\tx id}\\
\node{F(XZ\oplus YZ)}
\node{F(XZ)\oplus F(YZ)}\arrow{w,t}{\Fa}
\node{FXFZ\oplus FYFZ}\arrow{w,t}{\Fm \oplus \Fm}
\tag{7.2}
\end{diagram}\]

\indent A \emph{braided $Ann$-morphism} (or a \emph{homotopy}) 
$$u : (F, \Fa, \Fm, F_{\ast})\rightarrow (K, \Ka, \Km, K_{\ast})$$
between braided $Ann$-functors is an $\oplus$-morphism as well as  an $\otimes$-morphism.

If there exists a braided $Ann$-functor $(K,\Ka,\Km, K_{\ast}):\A'\rightarrow \A$ and  braided $Ann$-morphisms
$KF\stackrel{\sim}{\rightarrow} id_{\A}, FK \stackrel{\sim}{\rightarrow} id_{\A'}$, we say that $(F,\Fa,\Fm, F_{\ast})$ is a braided $Ann$-equivalence, and, $\A$ and $\A'$ are braided $Ann$-\emph{equivalent}.
\end{dn}

One can prove that each braided $Ann$-functor is a braided $Ann$-equivalence iff $F$ is a categorical equivalence.

For later use, we prove the lemma below.

\begin{lem}\label{lem31}
Every braided $Ann$-functor $F=(F, \Fa, \Fm, F_{\ast}): \A\rightarrow \A'$ is homotopic to a braided $Ann$-functor
$F'=(F', \Fa', \Fm', F_{\ast}'),$ where $F'1=1',$ and $F'_{\ast}=id_{1'}$.
\end{lem}
\begin{proof}
Consider the family of isomorphisms in $\A'$:
$$\theta_X=\begin{cases}id_{FX} \ \text{if}\ X\not= 1,\\ F_{\ast}^{-1}\ \  \text{if}\ \  X=1,\end{cases}$$
with $X\in \A$. Then $F$ can be deformed to a new braided $Ann$-functor $F'$, in the unique way such that $\theta: F\rightarrow F'$ becomes a homotopy, where
\begin{eqnarray*}
F'X=\begin{cases} FX \ \text{if}\ X\not= 1,\\ 1' \ \text{if}\  X=1;\end{cases}
\ F'(f: X\rightarrow Y)=(\theta_Y F(f)(\theta_X)^{-1}: F'X\rightarrow F'Y);\\
\Fa'_{X,Y}=\theta_{X\oplus Y}\Fa_{X,Y}(\theta_X\oplus \theta_Y)^{-1}; \Fm'_{X,Y}=\theta_{XY}\Fm_{X,Y}(\theta_X\otimes\theta_Y)^{-1};
 F_{\ast}'=\theta_1 F_{\ast}=id_1.
\end{eqnarray*}
\end{proof}

Thanks to Lemma \ref{lem31}, we can denote a braided $Ann$-functor by $(F,\Fa,\Fm)$ when we need not mention to $F_{\ast}$. 

\begin{pro}
In the definition of a braided $Ann$-functor, the commutation of the diagram (\ref{bd8.2}) is deduced from the other conditions.
\end{pro}

\begin{proof}
We consider the following diagram:
\[\scriptsize\setlength\unitlength{0.5cm}
\begin{picture}(20,11)
\put(0, 10){$F(ZX\ts ZY)$}
\put(0, 7){$F(XZ\ts YZ)$}
\put(0, 4){$F((X\ts Y)Z)$}
\put(0, 1){$F(Z(X\ts Y))$}

\put(6.5, 10){$F(ZX)\ts F(ZY)$}
\put(6.5, 7){$F(XZ)\ts F(YZ)$}
\put(6.5, 4){$F(X\ts Y)F(Z)$}
\put(6.5, 1){$F(Z) F(X\ts Y)$}

\put(14, 10){$F(Z)F(X)\ts F(Z)F(Y)$}
\put(14, 7){$F(X)F(Z)\ts F(Y)F(Z)$}
\put(14, 4){$(F(X)\ts F(Y))F(Z)$}
\put(14, 1){$F(Z)(F(X)\ts F(Y))$}

\put(2.5, 7.5){\vector(0, 1){2.1}}\put(0, 8.5){$F(c\ts c)$}
\put(2.5, 4.7){\vector(0, 1){2.1}}\put(2.7, 5.5){$F(\Rh)$}
\put(2.5, 3.7){\vector(0, -1){2.1}}\put(1.2, 2.5){$F(c)$}

\put(8.5, 7.7){\vector(0, 1){2.1}}\put(8.7, 8.5){$F(c)\ts F(c)$}
\put(8.5, 3.7){\vector(0, -1){2.1}}\put(7.9, 2.5){$c'$}

\put(16.5, 7.5){\vector(0, 1){2.1}}\put(17, 8.5){$c'\ts c'$}
\put(16.5, 4.6){\vector(0, 1){2.1}}\put(16.7, 5.5){$\Rh'$}
\put(16.5, 3.7){\vector(0, -1){2.1}}\put(16.7, 2.5){$c'$}

\put(6.1, 10.1){\vector(-1,0){2.5}}\put(4.75, 10.3){$\Fa$}
\put(13.7, 10.1){\vector(-1,0){2.9}}\put(11.53, 10.3){$\Fm\ts \Fm$}

\put(6.1, 7.1){\vector(-1,0){2.5}}\put(4.75, 7.3){$\Fa$}
\put(13.7, 7.1){\vector(-1,0){2.9}}\put(11.53, 7.3){$\Fm\ts \Fm$}

\put(6.1, 4.1){\vector(-1,0){2.5}}\put(4.75, 4.3){$\Fm$}
\put(13.7, 4.1){\vector(-1,0){2.9}}\put(11.53, 4.3){$\Fa\tx id$}

\put(6.1, 1.1){\vector(-1,0){2.5}}\put(4.75, 1.3){$\Fm$}
\put(13.7, 1.1){\vector(-1,0){2.9}}\put(11.53, 1.3){$id\tx \Fa$}

\put(-2.2,10.1){\vector(1,0){1.9}}
\put(-2.2,10.1){\line(0,-1){9}}
\put(-0.2,1.1){\line(-1,0){2}}\put(-2,5.5){$F(\Lh)$}

\put(21,10.1){\vector(-1,0){0.9}}
\put(21,10.1){\line(0,-1){9}}
\put(19.5,1.1){\line(1,0){1.5}}\put(21.2, 5.5){$\Lh'$}

\put(1, 5,5){(I)}
\put(5.4, 8.5){(II)}
\put(12.4, 8.5){(III)}
\put(9, 5.5){(IV)}
\put(5.4, 2.5){(V)}
\put(12.4, 2.5){(VI)}
\put(18, 5.5){(VII)}
\end{picture}\]

In the above diagram, the regions (I) and (VII) commute thanks to the diagram (\ref{bd5}), the region (II) commutes thanks to the naturality of $\Fa$, the region (III) commutes thanks to the naturality of $\Fm$, the region (V) commutes since $(F, \Fm)$ is a braided tensor functor, the region  (VI) commutes thanks to the naturality of $c'$, the perimeter commutes thanks to the diagram (\ref{bd8.1}). Hence, the region (IV) commutes, i.e., the region (\ref{bd8.2}) commutes.
\end{proof}

The structure transport for a monoidal category was presented by N. S. Rivano \cite{Ri}, H. X. Sinh  \cite{Si}. Then,   N. T. Quang \cite{Q3} extended naturally the structure transport for $Ann$-categories.  This transport can be extended to braided $Ann$-categories.

\begin{pro}\label{md21}
Let $(F,\Fa,\Fm): \A\ri \A'$ be an  $Ann$-equivalence  and $\A'$  is a braided $Ann$-category with a braiding $c'$. Then $\A$ becomes a braided $Ann$-category with a braiding $c$ given by the following commutative diagram

\[\scriptsize\begin{diagram}\label{bd9}
\node{F(X\tx Y)}\arrow{e,t}{F(c)}
\node{F(Y \tx X)}\\
\node{FX\tx FY}\arrow{e,t}{c'}\arrow{n,l}{\Fm}
\node{FY\tx FX}\arrow{n,r}{\Fm}
\tag{8}
\end{diagram}\]
and $(F,\Fa,\Fm)$ becomes a braided $Ann$-equivalence.
\end{pro}
\begin{proof}
We must prove that $c$ satisfies the diagrams (B1), (B2) and  (\ref{bd5}).

In order to prove that $c$ satisfies the diagram (B1), we consider the following diagram:
\[\scriptsize\setlength\unitlength{0.5cm}
\begin{picture}(21,17)
\put(0,16){$(FX\tx FY)\tx FZ$}
\put(0,13){$F(X\tx Y)\tx FZ$}
\put(0,10){$F((X\tx Y)\tx Z)$}
\put(0,7){$F(X\tx (Y\tx Z))$}
\put(0,4){$FX\tx F(Y\tx Z)$}
\put(0,1){$FX\tx (FY\tx FZ)$}

\put(7,16){$(FY\tx FX)\tx FZ$}
\put(7,13){$F(Y\tx X)\tx FZ$}
\put(7,10){$F((Y\tx X)\tx Z)$}
\put(7,7){$F((Y\tx Z)\tx X)$}
\put(7,4){$F(Y\tx Z)\tx FX$}
\put(7,1){$(FY\tx FZ)\tx FX$}

\put(14,16){$FY\tx (FX\tx FZ)$}
\put(14,13){$FY\tx F(X\tx Z)$}
\put(14,10){$F(Y\tx (X\tx Z))$}
\put(14,7){$F(Y\tx (Z\tx X))$}
\put(14,4){$FY\tx F(Z\tx X)$}
\put(14,1){$FY\tx (FZ\tx FX)$}

\put(2, 1.7){\vector(0,1){2}}\put(0,2.5){$id\tx \Fm$}
\put(2, 4.7){\vector(0,1){2}}\put(1,5.5){$\Fm$}
\put(2, 7.7){\vector(0,1){2}}\put(0,8.5){$F(a)$}
\put(2, 12.8){\vector(0,-1){2.2}}\put(1,11.5){$\Fm$}
\put(2, 15.8){\vector(0,-1){2.2}}\put(0,14.5){$ \Fm\tx id$}

\put(9.5, 1.7){\vector(0,1){2}}\put(7.5,2.5){$\Fm\tx id$}
\put(9.5, 4.7){\vector(0,1){2}}\put(8,5.5){$\Fm$}
\put(9.5, 12.8){\vector(0,-1){2.1}}\put(8,11.5){$\Fm$}
\put(9.5, 15.8){\vector(0,-1){2.1}}\put(7.2,14.5){$\Fm\tx id$}

\put(15.5, 1.7){\vector(0,1){2}}\put(15.7,2.5){$id\tx \Fm$}
\put(15.5, 4.7){\vector(0,1){2}}\put(15.7,5.5){$\Fm$}
\put(15.5, 7.7){\vector(0,1){2}}\put(12.5,8.5){$F(id\tx c)$}

\put(15.5, 12.8){\vector(0,-1){2}}\put(15.7,11.5){$\Fm$}
\put(15.5, 15.8){\vector(0,-1){2}}\put(15.7,14.5){$id\tx \Fm$}

\put(6.8, 1.2){\vector(-1,0){2}}\put(5.6,0.5){$c'$}
\put(13.8, 1.2){\vector(-1,0){2}}\put(12.6,0.5){$a'$}

\put(6.8, 4.2){\vector(-1,0){2.4}}\put(5.6,4.3){$c'$}

\put(6.8, 7.1){\vector(-1,0){2.3}}\put(5,7.3){$F(c)$}
\put(13.8, 7.1){\vector(-1,0){2.3}}\put(12,7.3){$F(a)$}

\put(6.8, 10.1){\vector(-1,0){2.4}}\put(4.3,10.4){$F(c\tx id)$}
\put(13.8, 10.1){\vector(-1,0){2.3}}\put(12,10.3){$F(a)$}

\put(6.8, 13.1){\vector(-1,0){2.3}}\put(4.3,13.4){$F(c)\tx id$}

\put(6.8, 16.1){\vector(-1,0){2}}\put(4.8,16.3){$c'\tx id$}
\put(13.8, 16.1){\vector(-1,0){2}}\put(12.5,16.3){$a'$}

\put(-0.2, 1.1){\line(-1,0){1.5}}
\put(-1.7, 1.1){\line(0,1){15}}
\put(-1.7, 16.1){\vector(1,0){1.5}}\put(-1.5,8.5){$a'$}

\put(18.5, 4.1){\line(1,0){1.5}}\put(17.2,11.5){$id\tx  F(c)$}
\put(20, 4.1){\line(0,1){9}}
\put(20, 13.1){\vector(-1,0){1.5}}

\put(18.7, 1.1){\line(1,0){3}}
\put(21.7, 1.1){\line(0,1){15}}
\put(21.7, 16.1){\vector(-1,0){3}}\put(19.8,14.5){$id\tx  c'$}

\put(-1,11.5){(I)}
\put(5,14.5){(II)}
\put(5,11.5){(III)}
\put(5,5.5){(IV)}
\put(5,2.5){(V)}

\put(12,14.5){(VI)}
\put(8.6,8.5){(VII)}
\put(12,2.5){(VIII)}
\put(18,8.5){(IX)}
\put(20.5,8.5){(X)}
\end{picture}\]

In the above diagram, the regions (I), (VI), (VIII) commute thanks to the determination of the morphism $a$, the regions (II), (IV), (X) commute thanks to the determination of the morphism $c$, the region (III) and (IX) commute thanks to the naturality of $\Fm$, the region (V) commutes thanks to the naturality of  $c'$, the perimeter commutes since  $(\A', \tx)$ is a braided tensor category. Hence, the region (VII) commutes. Since $F$ is an isomorphism, the diagram (B1) commutes.

Similarly, the diagram (B2) also commutes.

In order to prove that the diagram (\ref{bd5}) commutes, we consider the following diagram.

In that diagram, the region (I) commutes thanks to the naturality of $c'$, the regions (II) and (VII) commute since $(F,\Fm)$ is a braided tensor functor respect to $\tx$, the region (III) commutes thanks to the determination of $\Lh$, the region (V) thanks to the determination of $\Rh$, the region  (VI) commutes since $(F,\Fa)$ is an $\ts$-functor, the perimeter commutes since $\A'$ is an $Ann$-category. Hence, the region (IV) commutes. Since $F$ is an equivalence and from the commutation of the region (IV), the diagram (\ref{bd5}) commutes. 


\[\scriptsize\setlength\unitlength{0.5cm}
\begin{picture}(22,19)
\put(2.5,18){$FA.(FX\ts FY)$}
\put(4,17.8){\vector(0,-1){2.2}} \put(4.4,16.5){$id\tx \Fa$}

\put(14,18){$FA.FX\ts FA.FY$}
\put(6.7,18.1){\vector(1,0){7}}\put(10,18.3){$\Lh'$}
\put(15.3,17.8){\vector(0,-1){2.2}} \put(13,16.5){$\Fm\ts \Fm$}

\put(2.5,15){$FA.F(X\ts Y)$}
\put(4,14.8){\vector(0,-1){2.2}}\put(3,13.5){$\Fm$}

\put(14,15){$F(AX)\ts F(AY)$}
\put(15.3,14.8){\vector(0,-1){2.2}}\put(14,13.5){$\Fa$}

\put(2.5,12){$F(A.(X\ts Y))$}

\put(14,12){$F(AX\ts AY)$}
\put(6.3,12.1){\vector(1,0){7.4}} \put(9.5,12.3){$F(\Lh)$}

\put(2.5,9){$F((X\ts Y).A)$}
\put(14,9){$F(XA\ts YA)$}
\put(6.3,9.1){\vector(1,0){7.4}}
\put(9.5,9.3){$F(\Rh)$}
\put(4,9.5){\vector(0,1){2.3}}
\put(4.5,10.5){$F(c)$}
\put(15.3,9.5){\vector(0,1){2.3}}
\put(12.5,10.5){$F(c\ts c)$}

\put(2.5,6){$F(X\ts Y).FA$}
\put(4,6.5){\vector(0,1){2.3}}\put(3,7.7){$\Fm$}

\put(14,6){$F(XA)\ts F(YA)$}

\put(15.3,6.5){\vector(0,1){2.3}}\put(14,7.7){$\Fa$}

\put(2.5,3){$(FX\ts FY).FA$}
\put(14,3){$FX.FA\ts FY.FA$}
\put(7,3.1){\vector(1,0){6.7}}\put(10,3.3){$\Rh$}
\put(4,3.5){\vector(0,1){2.3}}\put(1.5,4.7){$\Fa\tx id$}
\put(15.3,3.5){\vector(0,1){2.3}}\put(13,4.7){$\Fm\ts \Fm$}

\put(18.5,15.2){\line(1,0){1.5}}
\put(20,15.2){\line(0,-1){9}}
\put(20,6.2){\vector(-1,0){1.5}}

\put(2.2,15.2){\line(-1,0){1.5}}
\put(0.7,15.2){\line(0,-1){9}}
\put(0.7,6.2){\vector(1,0){1.5}}

\put(21.5,3.1){\vector(-1,0){2.7}}
\put(21.5,3.1){\line(0,1){15}}
\put(21.5,18.1){\line(-1,0){2.7}}

\put(-1,3.1){\vector(1,0){3}}
\put(-1,3.1){\line(0,1){15}}
\put(-1,18.1){\line(1,0){3}}

\put(-0.8,11.5){$c'$}
\put(0.2,10.5){$c'$}

\put(19.3,10.5){$t_1 $}

\put(22,10.5){$t_2 $}

\put(0,16.5){(I)}
\put(1.5,10.5){(II)}
\put(9.5,15){(III)}
\put(9.5,10.5){(IV)}

\put(16.7,10.5){(VI)}
\put(18.5,16.5){(VII)}

\put(9.5,6){(V)}

\put(0,1){\text{where}\ \ \ $t_1=F(c_{A,X})\oplus F(c_{A,Y})$, \qquad \qquad $t_2=c'_{FA,FX}\oplus c'_{FA, FY}.$}
\end{picture}\]

\end{proof}

In Section \ref{sec5}, we shall use the structure transport to construct a {\it reduced} braided $Ann$-category of a braided $Ann$-category.


\section{Reduced braided $Ann$-categories}\label{sec5}

Let $\A$ be a braided $Ann$-category with the family of constraints
$$(a^+,c^+,(0,g,d), a,c,(1,l,r), \Lh, \Rh).$$

Then, according to \cite{Q3}, the set of the iso-classes of objects $\pi_0\A$ of $\A$ is a ring with the operations
$+,\times$, induced by $\oplus, \otimes$ in $\A$, and $\pi_1\A = Aut(0)$ is a $\pi_0\A$-bimodule with the actions:
\[su=\lambda_X(u),\ \ \ \ us=\rho_X(u),\]
where $X\in s, \ s\in \pi_0\A,\ u\in \pi_1\A$, and the functions $\lambda, \rho$ are defined as follows:
\[\scriptsize\begin{diagram}
\node{X.0}\arrow{e,t}{\hat{L}^X}\arrow{s,l}{id\tx u}
\node{0}\arrow{s,r}{\lambda_X(u)}
\node[3]{0.X}\arrow{e,t}{\hat{R}^X}\arrow{s,l}{u\tx id}
\node{0}\arrow{s,r}{\rho_X(u)}\\
\node{X.0}\arrow{e,t}{\hat{L}^X}
\node{0}
\node{(9.1)}
\node[2]{0.X}\arrow{e,t}{\hat{R}^X}
\node{0}
\node{(9.2)}
\end{diagram}\]
where $\hat{L}^X$ (respectively $\hat{R}^X$) is the isomorphirm induces by the pair $(L^X, \La^X)$ (respectively $(R^X, \Ra^X)$) (for detail, see Proposition 1.3 \cite{Q12}).

Thanks to the braiding property of $\otimes$,  $\pi_0\A$ is a commutative ring. Moreover, the two-sided actions  of the ring $\pi_0\A$ over $\pi_1\A$ coincide. It is expressed in the following proposition.


\begin{pro}\label{md20}
In a braided $Ann$-category $\A$, $\lambda=\rho$.
\end{pro}

\begin{proof} 
Directly deduced from Proposition 4.2 \cite{Q2} and the naturality of the braiding $c$.
\end{proof}

Let $\A$ be a braided $Ann$-category. Let $s\A$ be the reduced $Ann$-category of $\A$. Following, we shall transport the braiding $c$ of $\A$ for $s\A$ to make $s\A$ become a braided $Ann$-category. 

First, let us recall the main steps of the construction of a reduced $Ann$-category $s\A$ of an $Ann$-category $\A$, thanks to the structure transport (for detail, see  \cite{Q3}).  Objects of $s\A$ are the elements of $\pi_0\A$,
morphisms are automorphisms $(r,a): r\rightarrow r,\ r\in \pi_0\A, a\in \pi_1\A$. The composition law of two morphisms is defined by
\[(r,a)\circ (r,b)=(r,a+b).\]

In $\A$, we choose the representatives $X_s,\ s\in \pi_0\A$ such that $X_0=0, X_1=1$, and a family of isomorphisms $i_X: X\rightarrow X_s$ such that $i_{X_s}=id_{X_s}$. Then, we can determine two functors:\\

\quad{\small  \noindent\parbox{5cm}
{$G: \A\ri s\A,$\\
$G(X)=[X]=s,$\\
$G(f)=(s, \gamma_{X_s}^{-1}(i_Yfi_X^{-1})),$}
\qquad\qquad
\parbox{5cm}
{$H:s\A\ri \A,$\\
$H(s)=X_s,$\\
$H(s,u)=\gamma_{X_s}(u),$}}

\noindent
for $X,Y\in s$ and $f: X\ri Y$, and $\gamma_X$ is determined by the following commutative diagram:

\[\scriptsize\begin{diagram}\label{d9}
\node{X}\arrow{e,t}{\gamma_X(u)}
\node{X}\\
\node{0\ts X}\arrow{n,l}{g_{_{X}}}\arrow{e,t}{u\ts id}
\node{0\ts X}\arrow{n,r}{g_X}\tag{10}
\end{diagram}\]

\noindent Then, the operations on $s\A$ are defined:
\begin{eqnarray*} 
s\ts t&=&G(H(s)\ts H(t))=s+t,\\
s\tx t&=& G(H(s)\tx H(t))=st,\\
(s,u)\ts (t,v)&=&G(H(s,u)\ts H(t,v))=(s+t,u+v),\\
(s,u)\tx (t,v)&=&G(H(s,u)\tx H(t,v))=(st, sv+ut),
\end{eqnarray*}
with $s,t\in \pi_0\A$, $u,v\in \pi_1\A$. Obviously, they do not depend on the choice of the representative system $X_s, i_X.$

The constraints in $s\A$ are determined by the constraints in $\A$ by the concept of {\it stick}. A {\it stick} in $\A$ consists of a representatives
 $\left(X_s\right)_{s\in \pi_0\A}$ such that $X_0=0, X_1=1$ and isomorphisms 
\[\varphi_{s,t}: X_s\ts X_t\ri X_{s+t},\ \ \psi_{s,t}: X_sX_t\ri X_{st},\]
for all $s,t\in \pi_0\A,$ such that
\begin{eqnarray*}
\varphi_{0,t}=g_{X_t},\qquad  \varphi_{s,0}=d_{X_s},&\\
\psi_{1,t}=l_{X_t}, \qquad \psi_{s,1}=r_{X_s},& \psi_{0,t}=\Ro^{X_t},\quad \psi_{s,0}=\Lo^{X_s}.
\end{eqnarray*}

For the two operations $\ts,\tx$ of $s\A$, the unit constraints are chosen, respectively, as  $(0, id,id)$ and $(1,id, id)$. Put $\varphi=(\varphi_{s,t})$ and $\psi=(\psi_{s,t})$, we can define the constraints $\xi, \eta, \alpha, \lambda, \rho$ which are compatible, respectively, with the constraints $a^+, c^+, a, \frak{L}, \frak{R}$ of $\A$ via $H, \Ha=\varphi^{-1}, \Hm=\psi^{-1}, H_{\ast}=id_1$. Then
$$(\mathcal S, \xi,\eta, (0, id, id), \alpha, (1, id, id), \lambda, \rho)$$
is an  $Ann$-category which is equivalent to $\A$ by the $Ann$-equivalence $ H $.
Concurrently, the functor $G: \A\rightarrow s\A$ together with functor isomorphisms 
\begin{equation*}
\Ga_{X,Y}=G(i_X\ts i_Y),\quad \Gm_{X,Y}=G(i_X\tx i_Y),\quad G_{\ast}=id_1
\end{equation*} 
is an $Ann$-equivalence.

\begin{pro}\label{md23}
If  $\mathcal A$ is a braided $Ann$-category with the braiding $c$, then $s\A$ is a braided $Ann$-category with an induced braiding $c'=(\bullet, \beta)$ given by the following commutative diagram: 
\[\scriptsize\begin{diagram}
\node{X_r\tx X_s}\arrow{e,t}{\psi_{r,s}}\arrow{s,l}{c}
\node{X_{rs}}\arrow{s,r}{\gamma_{_{X_{rs}}}(\beta(r,s))}\\
\node{X_s\tx X_r}\arrow{e,t}{\psi_{s,r}}
\node{X_{sr}}
\end{diagram}\]
 Moreover, $(H, \Ha, \Hm)$, $(G,\Ga,\Gm)$ are braided $Ann$-equivalences.
\end{pro}
\begin{proof}
Directly deduced from Proposition \ref{md21}.
\end{proof}
We call  $s\A$ a {\it reduced braided $Ann$-category} of $\mathcal A$, and  $G, H$ {\it canonical braided Ann-equivalences}.
A reduced braided $Ann$-category $s\A$ of the braided $Ann$-category $\mathcal A$ is called a {\it braided $Ann$-category of the type $(R,M)$} if $\pi_0\mathcal A$ is replaced by $R$ and $\pi_1\mathcal A$ by $M$.

We now describe  specifically the constraints of the braided $Ann$-category $s\A$ of the type $(R,M)$.
We call the family $(h, \beta)$, where $h=(\xi, \eta, \alpha, \lambda, \rho)$, a {\it structure} of $s\A$.

\begin{pro}
In a braided $Ann$-category of the type $(R,M)$, the unit constraints are strict, and the family of other constraints is associated with a structure $(h, \beta)\in Z^3_{ab}(R, M)$.
\end{pro}

\begin{proof} Directly deduced from the axiomatics of a braided $Ann$-category.
\end{proof}

\begin{Note}
If $R$ trivially acts over $M$, the relations 
(S13), (S14) prove that $(\alpha, \beta)$ is an abelian 3-cocycle of the additive group $R_+$ with coefficients in $M$ of group  $H^{3}_{ab}(R_+, M)$ in the sense of  Eilenberg - Mac Lane \emph{(\cite{EM}, \cite{Mac1})}.
\end{Note}

We now consider the independence of induced constraints over $s\A$ if we change the stick.

\begin{pro}\label{md31}
If $s\A$ and $s\A'$ are two reduced braided $Ann$-categories of a braided $Ann$-category $\A,$ with, respectively, two sticks  $(X_s, \varphi, \psi)$, $(X'_s, \varphi', \psi')$,  then:
\begin{enumerate}
\item[\emph{(i)}] There exists a braided $Ann$-equivalence   $(F, \Fa, \Fm): s\A'\ri s\A$, where $F=id$.
\item[\emph{(ii)}] Two structures  $(h, \beta)$ and $(h', \beta')$  of , respectively, $s\A$ and $s\A'$ belong to the same cohomology class of the commutative ring $R$ with coefficients in $R$-module $M$.
\end{enumerate}
\end{pro}
\begin{proof}
(i) Call $G:\mathcal A\rightarrow  s \mathcal A$ and $H': s'\mathcal A\rightarrow \mathcal A$ canonical braided $Ann$-equivalences corresponding to $(X_s, \varphi, \psi)$, $(X'_s, \varphi', \psi')$. The composition $F=GH': s'\mathcal A\rightarrow  s\mathcal A$ is a braided $Ann$-equivalence. The underlying categories of $s\mathcal A$ and $s'\mathcal A$ are the same, and it is easy to see that $F=id_{s\mathcal A}$ are functors.

(ii) Setting $\mu=\Breve{GH'}, \nu=\widetilde{GH'}$. Then, from the compatibility of the braided $Ann$-functor $(F=id, \mu, \nu)$ with, respectively, the constraints of $s\A$ and $s\A'$,  we deduce the required relations.
\end{proof}



Clearly, we have
\begin{hq}\label{md32}
Two structures $(h,\beta), (h', \beta')$ are cohomologous iff they are two families of constraints of a braided $Ann$-category of the type $(R,M)$ which are compatible with each other respect to the braided $Ann$-equivalence $(F, \Fa, \Fm)$, where $F=id_{\A}$.
\end{hq}


We call a {\it trace} of the structure  $(h,\beta)$ is a map 
\begin{eqnarray*}
R&\rightarrow& M,\\
x&\mapsto& \beta(x,x).
\end{eqnarray*}

\begin{hq}
If $(h, \beta)$ and $(h', \beta')$ are two cohomologous structures, then \linebreak $[h]=[h']\in H^{3}_{MacL}(R, M)$ and two braidings $\beta, \beta'$ have the same trace.
\end{hq}


In Section \ref{sec7}, we present the third invariant of a braided $Ann$-category.


\section{Classification of braided $Ann$-functors of the type $(p, q)$}\label{sec6}

We now prove that each braided $Ann$-functor $(F, \Fa,\Fm):\A\ri\A^{'}$ induces a braided $Ann$-functor $sF$ on their reduced braided $Ann$-categories. Throughout this section, let $\mathcal S, \mathcal S'$  denote the braided $Ann$-categories form $(R,M,h,\beta),$ $ (R',M',h',\beta')$. 

\begin{dn}
 A functor  $F: \mathcal S\ri \mathcal S'$ is called a \emph{functor of the type $(p, q)$} if 
$$F(x)=p(x),\ \ \ \ F(x,a)=(p(x), q(a)), $$ 
where $p:R\ri R'$ is a ring homomorphism and  $q:M\ri M'$ is a group homomorphism satisfying 
$$q(xa)=p(x)q(a),\quad x\in R, a\in M.$$ 
\end{dn}

$M'$ can be regarded as an $R$-module with the action $sa'=p(s)a'$, then $q$ is an $R$-module homomorphism. We say that $(p,q)$ is a {\it pair of homomorphisms}.

\begin{pro}\label{s61}
Each braided $Ann$-functor from $\mathcal S$ to $\mathcal S'$ is a functor of the type $(p, q).$
\end{pro}

\begin{proof}
Since each braided $Ann$-functor is of course an $Ann$-functor and from Proposition 4.3 \cite{Q10}, the proof is completed.
\end{proof}

\begin{pro}\label{md22}
Let $\A$ and $\A'$ be two braided $Ann$-categories. Then each braided $Ann$-functor $(F,\Fa,\Fm, F_\ast):\A\ri \A'$ induces a braided $Ann$-functor $sF:s\A\ri s\A'$ of the type $(p, q),$ where
$$p=F_0:\pi_0\A\ri \pi_0\A'\ ;\ [X] \mapsto [FX],$$
$$q=F_1:\pi_1\A\ri\pi_1\A'\ ;\ u \mapsto \gamma_{F0}^{-1}(Fu),$$
where $\gamma$ is the map determined  by the diagram (\ref{d9}), with the following properties:
\begin{enumerate}
\item[(i)] $F$ is an equivalence iff $F_0, F_1$ are isomorphisms.
\item[(ii)] The braided $Ann$-functor $sF$ satisfies the following commutative diagram:
\[\scriptsize\begin{diagram}
\node{\A}\arrow{e,t}{F}
\node{\A'}\arrow{s,r}{G'}\\
\node{s\A}\arrow{n,l}{H}\arrow{e,t}{sF}
\node{s\A'}
\tag{11}
\end{diagram}\] 
where $H, G'$ are canonical braided $Ann$-equivalences.
\end{enumerate}
\end{pro}
\begin {proof}
(i) Call $sF$ the composition $G'\circ F\circ H.$ Therefore, $sF$ is a braided Ann-functor. One can verify that (see Theorem 4.6 \cite{Q3} for detail) $sF$ is of the type $(F_0, F_1)$.

\noindent (ii) According to Proposition 4.2 \cite{Q10}, we infer that the functor $sF$ makes the diagram (11) commute. Besides, since $F, H, G'$ are braided $Ann$-functors, so is $sF.$
\end {proof}

\begin{Note}\label{nx52}
If $F=(F,\Fa,\Fm,F_{\ast}):\A\rightarrow \A'$ is a braided $Ann$-functor,  the induced braided $Ann$-functor $sF:s\A\rightarrow s\A'$ satisfies  $sF(1)=1'$ and $ (sF)_{\ast}=(1', \gamma_{1'}^{-1}(F_{\ast}\circ i_{F1}^{-1}))$. Moreover, if $F1=1'$ and $F_{\ast}=id_{1'}$, then $(sF)_{\ast}=id_{1'}.$
\end{Note}

Since $\Fa_{x,y}=(\bullet, \mu(x,y)), \, \Fm_{x,y}=(\bullet, \nu(x,y)),$ we call $g_F$  the {\it pair of associated maps} with $(\Fa, \Fm),$ and the braided $Ann$-functor $F:\mathcal S\rightarrow\mathcal S'$ can be considered as a triplet $(p,q,g_F)$.
Thanks to  the compatibility of $F$ with the constraints, we infer
\begin{equation}
p^{\ast}(h',\beta')-q_{\ast}(h,\beta)=\partial (g_F), \tag{12}\label{7}
\end{equation}
where $p^{\ast}, q_{\ast}$ are canonical homomorphisms
$$Z^{3}_{ab}(R, M)\stackrel{q_{\ast}}{\longrightarrow}Z^{3}_{ab}(R, M')\stackrel{p^{\ast}}{\longleftarrow} Z^{3}_{ab}(R', M').$$

Moreover, two braided $Ann$-functors $(F,g_F), (F',g_{F'})$ are homotopic iff \linebreak
 $F'=F$, it means that they are of the same type $(p,q)$, and there exists a map $t:R\rightarrow M'$, such that
\begin{equation}
  g_{F'}=g_F+\partial t.\tag{13}\label{12}
\end{equation}  

We denote the set of homotopy classes  of braided $Ann$-functors of the type
$(p,q)$ from $\mathcal S$ to $\mathcal S'$ by
 $$Hom_{(p, q)}^{BrAnn}[\mathcal S, \mathcal S'].$$

If $F:\mathcal S\rightarrow \mathcal S'$ is a functor of the type $(p,q)$, the function 
$$k=p^{\ast}(h',\beta')-q_{\ast}(h,\beta)$$
is called {\it an obstruction} of the functor of the type $(p, q).$


\begin{thm}\label{s63}
The functor  $F:\mathcal S\ri \mathcal S' $ of the type  $(p, q)$ is a braided $Ann$-functor iff  its obstruction $\overline{k}=0$ in  $H_{ab}^3(R, M')$. Then, there exist bijections:
\emph{(i)}  
\begin{equation}\label{14} Hom_{(p, q)}^{BrAnn}[\mathcal S, \mathcal S']\leftrightarrow H^2_{ab}(R, M')\tag{14};
\end{equation}
\emph{(ii)} \begin{equation*} Aut(F)\leftrightarrow Z^1_{ab}(R, M').
\end{equation*}
\end{thm}

\begin {proof}
Let $F:\mathcal S\ri \mathcal S' $  be a braided $Ann$-functor of the type $(p, q)$. 
Then, from the equation (\ref{7}), we infer the obstruction of  $F$ vanishes in the group $H_{ab}^3(R, M')$.

Conversely, assume that the obstruction of the functor $F$  vanishes in the group $H_{ab}^3(R, M')$. Then, there exists a 2-cochain  $g=(\mu, \nu)$ such that $k=\partial g$, it means that 
$$p^{\ast}(h',\beta')-q_{\ast}(h,\beta)=\partial g.$$
Take $\Fa, \Fm$ be functor morphisms which are associated with the maps $\mu, \nu$, one can verify that  $(F,\Fa,\Fm)$ is a braided $Ann$-functor.

(i)  According to Theorem  4.5 \cite{Q10}, we have a bijection
$$\Phi: Hom^{Ann}_{(p,q)}[\mathcal S, \mathcal S']\ri H^2_{MacL}(R,M').$$

The bijection $\Phi$ is defined as follows (for detail, see Theorem 4.5 \cite{Q10}).

Since there exists an $Ann$-functor $(F, \Fa, \Fm): \mathcal S\rightarrow \mathcal S'$, we have 
\[p^{\ast}h'-q_{\ast}h=\partial_{MacL}g_F.\]
Let $g_F$ be fixed. Assume that $(K,\Ka, \Km): \mathcal S\rightarrow \mathcal S'$ is a braided $Ann$-functor of the type $(p,q)$. Then, 
\[p^{\ast}h'-q_{\ast}h=\partial_{MacL} g_{K}.\]
Hence,  $g_F-g_{K}$ is a 2-cocycle. The map $\Phi$ is given by:
\[\Phi([K])=[g_F-g_{K}].\]
Now, if $F, K$ are braided $Ann$-functors, then the above mentioned functions $g_F, g_K$ are 2-cochains of the commutative ring  $R$ with coefficients in the $R$-module $M$ and $g_F-g_K\in Z^2_{ab}(R,M')$. 
Besides,  since $Z^2_{ab}(R,M')\subset Z^2_{MacL}(R,M')$ and $B^2_{ab}(R,M')=B^2_{MacL}(R,M')$, we have  
$$H^2_{ab}(R,M')\subset H^2_{MacL}(R,M').$$
Call $\Phi'$ the restriction of the map $\Phi$ on the set $Hom^{BrAnn}_{(p,q)}[\mathcal S, \mathcal S'],$ we obtain an injection 
\[\Phi':Hom^{BrAnn}_{(p,q)}[\mathcal S, \mathcal S'] \rightarrow H^2_{ab}(R, M').\]

Now, assume that $g$ is an arbitrary 2-cocycle in $Z^2_{ab}(R,M')$. We have
\[\partial(g_F-g)=\partial g_F-\partial g=\partial g_F=p^{\ast}(h', \beta')-q_{\ast}(h,\beta).\]
Thus, there exists a braided $Ann$-functor $(K, \Ka, \Km): \mathcal S\rightarrow \mathcal S'$ of the type $(p,q)$, where the isomorphisms $\Ka, \Km$ are associated with the 2-cochain $g_F-g$. Clearly, 
$\Phi'([K])=[g]$, it means that $\Phi'$ is a surjection. Thus, $\Phi'$ is a bijection.

(ii) In the equation  (\ref{12}), with $F'=F$, we infer $\partial t=0$, i.e., $t\in Z^1_{ab}(R,M')$.
Thus, there exists a map
\begin{eqnarray*}
Aut(F)&\rightarrow& Z^1_{ab}(R,M'),\\
u&\mapsto& t.
\end{eqnarray*}
It is easy to see  that the above map is a bijection.
\end{proof}

\section{ Classification Theorems}\label{sec7}

Let $\bf BrAnn$ denote the category whose objects are braided $Ann$-categories $\mathcal A$, and whose morphirms are braided $Ann$-functors between braided $Ann$-categories.

Similar to the Classification Theorem for graded {\it Picard} categories (Theorem 3.12 [3]), we determine ${\bf H^{3}_{BrAnn}}$ to be a category whose objects are the triplets 
 $(R, M, [h,\beta])$, where $[h, \beta]\in H^3_{ab}(R,M)$ and $(R,M, h, \beta)$ is a braided $Ann$-category. A morphism 
$(p,q):(R,M,[h,\beta])\ri (R',M',[h',\beta'])$ 
of $\bf H^{3}_{BrAnn}$  is a pair  $(p, q)$ such that there exists  $g=(\mu,\nu):R^2\rightarrow M'$ and $(p,q,g)$ is a braided $Ann$-functor  $(R,M,h,\beta)\ri (R',M',h',\beta')$, i.e., $[p^\ast (h',\beta')]=[q_\ast (h,\beta)]\in H^3_{ab}(R,M')$. The composition law in $\bf H^{3}_{BrAnn}$ is defined by
$$(p', q')\circ (p,q)=(p'p, q'q).$$
Note that, {\it two braided $Ann$-functors $F, F':\mathcal A \rightarrow \mathcal A'$ are homotopic iff $F_i=F_i', i=0,1$ and $[g_F]=[g_{F'}]$}. Let
$$Hom_{(p,q)}^{BrAnn}[\mathcal A,\mathcal A']$$
denote the set of homotopy classes of braided $Ann$-functors  $\mathcal A\ri\mathcal A'$ which induce the same pair $(p,q)$, we assert the main theorem of this section.

\begin{thm}{\emph{(The classification theorem)}}
There exists a functor
\begin{eqnarray*}
 d:{\bf BrAnn} &\rightarrow& \bf H^{3}_{BrAnn},\\
\mathcal A     &\mapsto&       (\pi_0\mathcal A, \pi_1\mathcal A, [(h,\beta)_{\mathcal A}]),\\
F=(F, \Fa, \Fm)  &\mapsto& (F_0, F_1),
\end{eqnarray*}
with the following properties:
\begin {enumerate}
\item[\emph{(i)}]   $dF$ is an isomorphism iff $F$ is an  equivalence;
\item[\emph{(ii)}]  $d$ is a surjection over the set of objects;
\item[\emph{(iii)}] $d$ is  full but not faithful. Indeed, for any arrow $(p, q):d\mathcal A\rightarrow d\mathcal A'$ in $\bf H^{3}_{BrAnn}$, there is a bijection:
\begin{equation}\label{15}
d:Hom_{(p, q)}^{BrAnn}[\mathcal A, \mathcal A']\ri H^2_{ab}(\pi_0\mathcal A, \pi_1\mathcal A').\tag{15}
\end{equation}
\end {enumerate}
\end{thm}
\begin{proof}
In the braided $Ann$-category $\mathcal A$, for any stick $(X_s,i_X)$, we can build a reduced braided  $Ann$-category
$(\pi_0\mathcal A,\pi_1\mathcal A,h,\beta)$.
When the choice of  stick is changed, the 3-cocycle $(h, \beta)$ is replaced by $(h',\beta')$ which is in the same cohomology class with $(h,\beta)$. So $\mathcal A$ defines uniquely an element $[(h,\beta)]\in H^3(\pi_0\mathcal A,\pi_1\mathcal A)$. It means that $d$ is a map over the set of objects.

For the braided $Ann$-functors
$$\mathcal A\stackrel{F}{\longrightarrow}\mathcal A'\stackrel{F'}{\longrightarrow}\mathcal A''$$
\noindent it is easy to see that $d(F'\circ F)=(dF')\circ (dF)$, and  ${d}({id_\mathcal A})=id_{d\mathcal A}$. 
Thus $d$ is a functor.

(i) Thanks to Proposition \ref{md22}.

(ii) If $(R,M,[h,\beta])$ is an object of  $\bf{H}^3_{BrAnn}$, then 
$\mathcal S=(R,M,h,\beta)$ is a braided $Ann$-category of the type $(R,M)$ and obviously,
$d\mathcal S=(R,M,[h,\beta])$.

(iii) Assume that $(p,q)$ is a morphism of
$Hom_{\bf{H}^3_{BrAnn}}(d\mathcal A, d\mathcal A')$. Then, there exists 
a function $g=(\mu, \nu)$, $\mu, \nu:(\pi_0\mathcal{A})^2\ri \pi_1\mathcal{A'}$ such that
$$p^\ast{(h,\beta)}_{\mathcal A'}=q_\ast{(h,\beta)}_\mathcal A+\partial g.$$
So
$$K=(p,q,g):(\pi_0\mathcal A, \pi_1\mathcal A,(h,\beta)_\mathcal A)\ri
(\pi_0\mathcal A', \pi_1\mathcal A',(h,\beta)_{\mathcal A'})$$ is a braided $Ann$-functor.
Then, according to Proposition \ref{md22}, the composition of braided $Ann$-functors $F=H'\circ K\circ
G:\mathcal A\ri\mathcal A'$ which induces a braided $Ann$-functor $dF.$ 
It is easy to see that 
$dF=(p,q)$. This proves that  the functor $d$ is full.

To show that (\ref{15}) is a bijection, we prove that the map
\begin{align}\label{16}
s:Hom_{(p,q)}^{BrAnn}[\mathcal A,\mathcal A']&\ {\ri}\ Hom_{(p,q)}^{BrAnn}[{s}\mathcal A,{s}\mathcal A']\tag{16},\\
{[F]}&\mapsto [sF]\notag
\end{align}
is a bijection.

Clearly, if  $F, F':\mathcal A\rightarrow \mathcal A'$ are homotopic, then the reduced braided $Ann$-functors  $sF, sF'$ are homotopic. Conversely, if $sF, sF'$ are homotopic, then the compositions $E=H'(sF)G$ and $E'=H'(sF')G$ are homotopic. The braided $Ann$-functors $E, E'$ are homotopic to, respectively,  $F, F'$. So, $F$ and $F'$ are homotopic. This proves that $s$ is injective.

Now, if $K=(p,q,g):s\mathcal A\ri s\mathcal A'$ is a braided $Ann$-functor, then the composition 
$$F=H'\circ K\circ
G:\mathcal A\ri\mathcal A'$$ is a braided $Ann$-functor such that
$s F=K$, i.e., 
$s$ is a surjection. Now, the bijection (\ref{15}) is just the composition of (\ref{16}) and (\ref{14}).
\end{proof}

According to Proposition \ref{md22},  we may simplify the problem of equivalence classification of braided $Ann$-categories by the classification
of braided $Ann$-categories which have the same (up to an isomorphism) first two invariants.

Let $R$ be a commutative ring with the unit, $M$ be an $R$-module (and regarded as a ring with the  null multiplication). We say that a braided $Ann$-category $\mathcal A$ has a {\it pre-stick of the type $(R,M)$} if there exists a pair of ring isomorphisms  $\epsilon=(p,q)$,
\[p: R\ri \pi_0\mathcal A,\qquad q: M\ri \pi_1\mathcal A\]
which is compatible with the module action, i.e.,
\[q(su)=p(s)q(u),\]
for $s\in R, u\in M$. The pair  $(p, q)$ is called  {\it a pre-stick of the type} $(R, M)$ respect to the braided $Ann$-category  $\mathcal A$.

A {\it morphism} between two braided $Ann$-categories $\mathcal A, \mathcal A'$, whose pre-sticks (respectively, $\epsilon=(p, q), \epsilon'=(p', q')$) are of the type $(R, M)$ is a braided $Ann$-functor $(F,\Fa, \Fm): \mathcal A\ri \mathcal A'$ such that the following diagrams
\[\scriptsize\begin{diagram} 
\node{\pi_0\mathcal A}\arrow[2]{e,t}{ F_0}
\node[2]{\pi_0\mathcal A'}
\node[3]{\pi_1\mathcal A}\arrow[2]{e,t}{F_1}
\node[2]{\pi_1\mathcal A'}\\
\node[2]{R}\arrow{nw,b}{p}\arrow{ne,b}{p'}\node{(17.1)}
\node[4]{M}\arrow{nw,b}{q}\arrow{ne,b}{q'}\node{(17.2)}
\end{diagram}\] 
commute, where $F_0, F_1$ are two homomorphisms induced by $(F,\Fa,\Fm)$.

Clearly, from the definition, we infer that  $F_0, F_1$ are isomorphisms and thus  $F$ is an equivalence.

Let
$${\bf BrAnn}[R, M]$$
denote the set of iso-classes of braided $Ann$-categories whose pre-sticks are of the type $(R, M)$.  Using Proposition \ref{md31}, we can prove the following theorem.

\begin{thm}\label{md34}
There exists a bijection 
\begin{eqnarray*}
\Gamma:{\bf BrAnn}[R, M]&\rightarrow& H^3_{ab}(R, M),\\
{[\mathcal A]}&\mapsto&  q^{-1}_\ast p^\ast [(h, \beta)_{\mathcal A}].
\end{eqnarray*}
\end{thm}
\begin{proof}
According to Proposition \ref{md31}, each braided $Ann$-category $\A$ defines uniquely an element $[(h,\beta)_{\A}]\in H^3_{ab}(\pi_0\A,\pi_1\A)$, and so it defines an element
$$\epsilon[(h, \beta)_{\mathcal A}]=q^{-1}_\ast p^\ast [(h, \beta)_{\mathcal A}]\in H^3_{ab}(R, M).$$

Now, if $F:\A\ri \A'$ is a morphism of two braided $Ann$-categories and $F$ is of the pre-stick of the type $(p,q)$, the induced braided $Ann$-functor $sF=(p, q, g_F)$, satisfies the equation (\ref{7}), and so
$$p^\ast [(h,\beta)_{\mathcal A'}]=q_\ast [(h,\beta)_{\mathcal A'}].$$
Then, it is easy to deduce that
$$\epsilon'[(h, \beta)_{\mathcal A'}]=\epsilon[(h, \beta)_{\mathcal A}].$$
This proves that  $\Gamma$ is a map. Moreover, it is a injective. Indeed, if
$\Gamma(\mathcal A)=\Gamma(\mathcal A')$,  then
$$\epsilon'(h, \beta)_{\mathcal A'}-\epsilon(h, \beta)_{\mathcal A}=\partial g.$$
 So, there exists a braided $Ann$-functor $J$ of the type $(id,id)$ from \linebreak $\mathcal I=(R, M, \epsilon(h, \beta)_{\mathcal A})$ to $\mathcal I'=(R, M,\epsilon'(h, \beta)_{\mathcal A'})$. The composition
$$\mathcal A\stackrel{G}{\longrightarrow}s\mathcal A\stackrel{\epsilon^{-1}}{\longrightarrow} \mathcal I\stackrel{J}{\longrightarrow} \mathcal I'\stackrel{\epsilon'}{\longrightarrow}s\mathcal A'\stackrel{H'}{\longrightarrow}\mathcal A',$$
proves that $[\mathcal A]=[\mathcal A']$, and  $\Gamma$ is injective. Obviously, $\Gamma$ is a surjection.
\end{proof}

In a connection with Harrison cohomology  \cite{Har}, we obtain the following result.

\begin{hq}\label{md35}  
There exists an injection from the Harrison cohomology group $H^{3}_{Har}(R, M)$ to the set of iso-classes of the braided $Ann$-categories whose pre-sticks are of the type $(R,M)$.
\end{hq}
\begin{proof}
If $\alpha$ is a Harrison 3-cocycle of the commutative ring
$R$, with coefficients in $M$,  then it together with $\xi=0, \eta=0,  \beta=0, \lambda=0, \rho=0$ is a structure of the braided $Ann$-category $s(R,M)$. Moreover, if $\alpha'$ is another Harrison 3-cocycle, then $\alpha'-\alpha=\partial(g)$ iff braided $Ann$-categories $s(R,M,\alpha)$ and  
$s(R,M,\alpha')$ are equivalent.
\end{proof}


\begin{center}

\end{center}
\end{document}